\documentclass[10pt]{amsart}
\usepackage{amsmath}
\usepackage{cases}
\usepackage{mathrsfs}
\usepackage{bbm}
\usepackage{amssymb}
\usepackage{amscd}
\usepackage{amsfonts,latexsym,amsmath,amsthm,amsxtra,mathdots,amssymb,latexsym,mathabx}
\usepackage[all,cmtip]{xy}
\RequirePackage{amsmath} \RequirePackage{amssymb}
\usepackage{color}
\usepackage{colordvi}
\usepackage{multicol}
\usepackage{hyperref}
\usepackage{mathtools}
\usepackage[margin=1in]{geometry}
\usepackage{xcolor}
\hypersetup{
    colorlinks,
    linkcolor={red!50!black},
    citecolor={blue!50!black},
    urlcolor={blue!80!black}
}
\usepackage{cite}

     \newcommand{\BF}{{\mathbb {F}}}

     \newcommand{\BR}{{\mathbb {R}}}

     \newcommand{\BZ}{{\mathbb {Z}}}

     \newcommand{\CJ}{{\mathcal {J}}}
     \newcommand{\CL}{{\mathcal {L}}}
     \newcommand{\CN}{{\mathcal {N}}}

\def\-{^{-1}}

\newtheorem{Theorem}{Theorem}[section]
\newtheorem{Lemma}[Theorem]{Lemma}

\newtheorem{Remark}[Theorem]{Remark}
\newcommand{\sumstar}{\sideset{}{^\star}\sum}

\begin{document}

\title{The Burgess bound via a trivial delta method}

\author[K. Aggarwal, R. Holowinsky, Y. Lin, and Q. Sun]{Keshav Aggarwal, Roman Holowinsky, Yongxiao Lin, and Qingfeng Sun}


\address{Department of Mathematics, The Ohio State University\\ 231 W 18th Avenue\\
Columbus, Ohio 43210-1174}
\email{aggarwal.78@buckeyemail.osu.edu}
\email{holowinsky.1@osu.edu}
\email{lin.1765@buckeyemail.osu.edu}

\address{School of Mathematics and Statistics
\\ Shandong University, Weihai \\ Weihai \\Shandong 264209 \\China}
\email{qfsun@sdu.edu.cn}

\begin{abstract}
Let $g$ be a fixed Hecke cusp form for $\rm SL(2,\BZ)$ and $\chi$ be a primitive Dirichlet character of conductor $M$. The best known subconvex bound for $L(1/2,g\otimes \chi)$ is of Burgess strength. The bound was proved by a couple of methods: shifted convolution sums and the Petersson/Kuznetsov formula analysis. It is natural to ask what inputs are really needed to prove a Burgess-type bound on $\rm GL(2)$. In this paper, we give a new proof of the Burgess-type bounds ${L(1/2,g\otimes \chi)\ll_{g,\varepsilon} M^{1/2-1/8+\varepsilon}}$ and $L(1/2,\chi)\ll_{\varepsilon} M^{1/4-1/16+\varepsilon}$ that does not require the basic tools of the previous proofs and instead uses a trivial delta method.
\end{abstract}



\maketitle
\section{Introduction and statement of results}
Let $g$ be a fixed Hecke cusp form on $\rm GL(2)$ and let $\chi$ be a primitive Dirichlet character modulo $M$. Subconvex bounds for the twisted $L$-functions
\begin{equation*}
\begin{split}
L(s,g\otimes \chi)=\sum_{n=1}^{\infty}\frac{\lambda_g(n)\chi(n)}{n^s}
\end{split}\end{equation*}
have been of interest for a while, with lots of applications. The strongest known bound is of Burgess quality (in special cases better results than Burgess-type exponent are known). The first subconvex bound in the conductor aspect was established by Burgess \cite{Burgess} for Dirichlet $L$-functions $L(s,\chi)$. For a primitive Dirichlet character $\chi$ modulo $M$, Burgess proved
\begin{equation}\label{Burgess1}
\begin{split}
L\left(\frac{1}{2},\chi\right)\ll_{\varepsilon} M^{1/4-1/16+\varepsilon}.
\end{split}\end{equation}
A subconvex bound of such strength is called a Burgess bound. In the $\rm GL(2)$ setting, the Burgess bound for an $L$-function of a Hecke cusp form $g$ twisted by a primitive Dirichlet character $\chi$ of large conductor $M$  is
\begin{equation}\label{Burgess2}
\begin{split}
L\left(\frac{1}{2},g\otimes \chi\right)\ll_{g,\varepsilon} M^{1/2-1/8+\varepsilon}.
\end{split}
\end{equation}
The first subconvex bound for $L\left(1/2,g\otimes \chi\right)$ was obtained by Duke, Friedlander, and Iwaniec \cite{DFI}, who studied an amplified second moment and reduced the problem to treating shifted convolution sums of the form \begin{equation*}\sum_{\ell_1m-\ell_2n=h}\overline{\lambda_g(m)}\,\lambda_g(n),
\end{equation*}
for which they invented the $\delta$-symbol method in their name to deal with. Their approach gives \begin{equation*}L\left(\frac{1}{2},g\otimes \chi\right)\ll_{g,\varepsilon} M^{1/2-1/22+\varepsilon},
\end{equation*}
where $g$ is a holomorphic cusp form on the full modular group. Bykovski\u\i \cite{Bykovskii}, who embedded the $L$-function $L\left(1/2,g\otimes \chi\right)$ in question into a larger family $\sum_{g^\prime \in S(M)}\left|L(1/2,g^\prime\otimes \chi)\right|^2$ (and with amplification) and treated the later by applying the Petersson formula, where $S(M)$ is an orthogonal basis of the space of holomorphic cusp forms of level $M$, was able to sharpen the bound of Duke, Friedlander, and Iwaniec to the Burgess quality \eqref{Burgess2}. Bykovski\u\i's approach has the advantage of avoiding the treatment of shifted convolution sums and is regarded as the simplest subconvex proof for this case for a while. For more general cusp forms $g$ the bound \eqref{Burgess2} was established by Blomer, Harcos, and Michel \cite{Blo-Har-Mic} under the Ramanujan conjecture, and subsequently by Blomer and Harcos \cite{Blomer-Harcos} unconditionally. This latter bound was generalized by Fouvry, Kowalski, and Michel \cite{FKM} to more general algebraic twists including the Dirichlet characters $\chi$. 

It is then natural to ask: what inputs are really needed to prove a Burgess-type bound on $\rm GL(2)$? Is the shifted convolution sum treatment really needed? Is there anything special about the Petersson formula analysis? Or is subconvexity, even with the strongest known exponent, a soft feature? In asking these questions, we were able to find a proof that requires none of the tools required by the previous works. We came across this argument through a careful analysis of Munshi's $\rm GL(2)$ $\delta$-symbol proof in \cite{Munshi17}.

In a series of papers \cite{Mun1, Mun2}, Munshi introduced a novel $\rm GL(2)$ Petersson $\delta$-symbol method to prove level aspect subconvex bounds for $\rm GL(3)$ $L$-functions. In a recent paper \cite{Munshi17}, he demonstrated that the $\rm GL(2)$ $\delta$-symbol method can also be applied to the classical setting of Dirichlet $L$-functions as well as $L(s,g\otimes \chi)$ and re-established the bounds \eqref{Burgess1} and \eqref{Burgess2} simultaneously. 

While studying the works of Munshi, Holowinsky and Nelson \cite{HN17} discovered the following key identity hidden within Munshi's proof \cite{Mun2},
\begin{equation*}
\begin{split}
\chi(n)=\frac{M}{Rg_{\bar{\chi}}}\sum_{r=1}^{\infty}\chi(r)e\left(\frac{n\bar{r}}{M}\right)V\left(\frac{r}{R}\right)-\frac{1}{g_{\bar{\chi}}}\sum_{r\neq 0}S_{\chi}(r,n;M)\widehat{V}\left(\frac{r}{M/R}\right).
\end{split}
\end{equation*}
Here $g_\chi$ is the Gauss sum, $S_{\chi}(r,n;M)$ is the generalized Kloosterman sum, $\widehat{V}$ is the Fourier transform of the Schwartz
function $V$ which is supported on $[1,2]$ and is normalized such that $\widehat{V}(0)=1$, and $R>0$ is a parameter. This allowed them to produce a method which removed the use of the $\rm GL(2)$ $\delta$-symbol and establish a stronger subconvex bound. Subsequently, Lin \cite{Lin} was able to generalize the identity in the application to the subconvexity problem in both the Dirichlet character twist and $t$-aspect case via the identity 
\begin{equation*}
\begin{split}
\chi(n)n^{-it}V_A\left(\frac{n}{N}\right)=&\, O\left(t^{1/2-A}\right)+\bigg(\frac{2\pi}{Mt}\bigg)^{it} e\bigg(\frac{t}{2\pi}\bigg)\frac{M^{2}t^{3/2}}{Ng_{\bar{\chi}}}\sum_{r=1}^{\infty}\chi(r)e\bigg(\frac{n\bar{r}}{M}\bigg)r^{-it}e\bigg(-\frac{n}{Mr}\bigg)V\left(\frac{r}{N/Mt}\right)\\
-&\bigg(\frac{2\pi}{N}\bigg)^{it}e\bigg(\frac{t}{2\pi}\bigg) \frac{t^{1/2}}{g_{\bar{\chi}}}\sum_{r\neq 0}S_{\chi}(r,n;M)\int_{\mathbb{R}}V(x)x^{-it}e\bigg(-\frac{nt}{Nx}\bigg)e\bigg(-\frac{rNx}{M^2t}\bigg)\,\mathrm{d}x.
\end{split}
\end{equation*}
Here $n\asymp N$, $A\geq 1$ is any integer, $V_A(x)$ is a smooth compactly supported function with bounded derivatives.
With this approach, Lin \cite{Lin} obtained the following bound
\begin{equation*}
\begin{split}
L\left(\frac{1}{2}+it,\pi\otimes \chi \right)\ll_{\pi, \varepsilon}&(M(|t|+1))^{3/4-1/36+\varepsilon},
\end{split}
\end{equation*}
for $\pi$ being a fixed Hecke--Maass cusp form for $\mathrm{SL}(3,\mathbb{Z})$.

In this paper, we demonstrate that one is again able to remove the $\rm GL(2)$ $\delta$-symbol method and replace it in our subconvexity problem by the following trivial key identity,
\begin{equation*}
\delta(n=0)=\frac{1}{q}\sum_{c|q}\underset{\begin{subarray}{c}
a\bmod c\\ (a,c)=1 \end{subarray}}{\sum}e\left(\frac{an}{c}\right), \quad \mbox{when}\quad  q>|n|,
\end{equation*}
where $\delta(n=0)$ denotes the Kronecker delta symbol. We shall establish the following bounds.
\begin{Theorem}\label{cusp form thm}
Let $g$ be a fixed Hecke cusp form for $\rm SL(2,\mathbb{Z})$ and $\chi$ be a primitive Dirichlet character modulo a prime $M$. For any $\varepsilon>0$, 
\begin{equation*}
\begin{split}
L\left(\frac{1}{2},g\otimes \chi\right)\ll_{g,\varepsilon} M^{1/2-1/8+\varepsilon}.
\end{split}\end{equation*}
\end{Theorem}

\begin{Theorem}\label{divisor thm}
Let $\chi$ be a primitive Dirichlet character modulo a prime $M$. For any $\varepsilon>0$,
\begin{equation*}
\begin{split}
L\left(\frac{1}{2},\chi\right)\ll_{\varepsilon} M^{1/4-1/16+\varepsilon}.
\end{split}\end{equation*}
\end{Theorem}

Our approach is more direct. Neither the treatment of shifted convolution sums nor the use of the Petersson/Kuznetsov formula, is needed. We should point out that in Burgess' proof in \cite{Burgess} he appealed to estimates from algebraic geometry, which follow from Weil's proof of the Riemann hypothesis for curves over finite fields. In our case, the algebraic geometry input that we use is the work \cite{Adolphson-Sperber} of Adolphson and Sperber who proved their results in light of Deligne's general results. We also remark that better bounds were known for $\chi$ quadratic: Conrey and Iwaniec \cite{Conrey-Iwaniec} proved that the Weyl-type bound $L\left(\frac{1}{2},\chi\right)\ll_{\varepsilon} M^{1/6+\varepsilon}$ holds. This bound was recently extended to \emph{any} character $\chi$ by Petrow and Young \cite{Petrow-Young1, Petrow-Young2}. 

\section{Some Notations and Lemmas}

For a smooth function $V$ with bounded derivatives, we define its Fourier transform
\begin{equation*}
\widehat{V}(x) = \int_\BR V(u)e(-xu)\mathrm{d}u.
\end{equation*}
Repeated integration by parts shows 
\begin{equation}\label{Vcheck}
\widehat{V}(x)\ll_A (1+|x|)^{-A},
\end{equation}
for any $A\geq 0$.

Next, we collect some lemmas that we will use for the proof.

\begin{Lemma}[Trivial delta method]One has
\begin{equation}\label{trivial delta}
\delta(n\equiv m\bmod q)=\frac{1}{q}\sum_{c|q}\underset{\begin{subarray}{c}
a\bmod c \end{subarray}}{\sumstar}e\left(\frac{a(n-m)}{c}\right).
\end{equation}
where the star over the inner sum denotes the sum is over $(a,c)=1$.
\end{Lemma}

\begin{Lemma}[Voronoi summation formula, {\cite[Theorem A.4]{KMV2002}}]\label{voronoi}
Let $g$ be a Hecke cusp form of level $1$ with Fourier coefficients $\lambda_g(n)$. Let $c\in \mathbb{N}$ and $a\in \mathbb{Z}$ be such that $(a,c)=1$ and let $W$ be a smooth compactly supported function. For $N>0$,
\begin{equation*}\label{voronoi for tau}
\begin{split}
\sum_{n=1}^{\infty}\lambda_g(n)e\left(\frac{an}{c}\right)W\left(\frac{n}{N}\right)
=\mathcal{I}(g;W,c,N) + \frac{N}{c} \sum_{\pm}\sum_{n=1}^{\infty}\lambda_g(n)e\left(\mp\frac{\overline{a}n}{c}\right)\widecheck{W}_g^{\pm}\left(\frac{nN}{c^2}\right),
\end{split}
\end{equation*}
where 
\begin{equation*}
\mathcal{I}(g;W,c,N)=
\begin{cases}
\frac{N}{c}\int_0^\infty (\log xN+2\gamma-2\log c)W(x)\mathrm{d}x \quad &\textit{ if } \lambda_g \textit{ is the divisor function } \tau,\\
\qquad 0 \quad &\textit{ otherwise}.
\end{cases}
\end{equation*}
Here $\gamma$ is the Euler's constant. $\widecheck{W}_g^\pm$ is an integral transform of $W$ given by the following.
\begin{enumerate}
\item If $g$ is holomorphic of weight $k$, then
\begin{equation*}
\widecheck{W}^+_g(y) = \int_0^\infty W(x)2\pi i^k J_{k-1}(4\pi\sqrt{yx})\mathrm{d}x,
\end{equation*} 
and $\widecheck{W}^-_g=0$.

\item If $g$ is a Maass form with $(\Delta+\lambda)g=0$ and $\lambda=1/4+r^2$, and $\varepsilon_g$ is an eigenvalue under the reflection operator,
\begin{equation*}
\widecheck{W}^+_g(y) = \int_0^\infty \frac{-\pi W(x)}{\sin\pi ir} (J_{2ir}(4\pi\sqrt{yx}) - J_{-2ir}(4\pi\sqrt{yx})) \mathrm{d}x,
\end{equation*}
and
\begin{equation*}
\widecheck{W}^-_g(y) = \int_0^\infty 4\varepsilon_g\cosh(\pi r)W(x)K_{2ir}(4\pi\sqrt{yx}) \mathrm{d}x.
\end{equation*}
If $r=0$,
\begin{equation*}
\widecheck{W}^+_g(y) = \int_0^\infty -2\pi W(x) Y_0(4\pi\sqrt{yx}) \mathrm{d}x
\quad
\text{ and } 
\quad
\widecheck{W}^-_g(y) = \int_0^\infty 4\varepsilon_g W(x) K_0(4\pi\sqrt{yx}) \mathrm{d}x.
\end{equation*}

\item When $\lambda_g(n)=\tau(n)$ is the divisor function, 
\begin{equation*}
\widecheck{W}^+_g(y) = \int_0^\infty -2\pi Y_0(4\pi\sqrt{xy})\mathrm{d}x \quad \text{ and } \quad \widecheck{W}^-_g(y) = \int_0^\infty 4K_0(4\pi\sqrt{yx})\mathrm{d}x.
\end{equation*}
\end{enumerate}
\end{Lemma}
In each case, we have
\begin{equation}\label{What}
\widecheck{W}_g^\pm(x)\ll_A (1+|x|)^{-A},
\end{equation}
for any $A\geq 0$.

Very often we will use the following bound when $g$ is a Maass form with Fourier coefficients $\lambda_g(n)$.
\begin{Lemma}[Ramanujan bound on average] \label{Ramanujan}
Let $W$ be a smooth function with compact support contained in $\mathbb{R}_{>0}$, satisfying $W^{(j)}(x)\ll_j 1$. Then
\begin{equation}\label{Rankin--Selberg}
\sum_{n=1}^{\infty}\left|\lambda_g(n)\right|W\left(\frac{n}{X}\right)\ll_{g} X.
\end{equation}
\end{Lemma}
This follows from the Cauchy--Schwarz inequality and the Rankin--Selberg estimate $\sum_{n\leq X}\left|\lambda_g(n)\right|^2\ll_g X$ (see \cite{Molteni}).

\section{Sketch of the proof}
Since there are quite a few auxiliary parameters involved in our proof, some are essential while others are not, in this section we provide a quick sketch to guide the reader through the essential part of our argument. 

For $\chi \bmod M$, we start with the following sum
\begin{equation*}
S(N):=\sum_{n\sim N}\lambda_g(n)\chi(n),
\end{equation*}
 where the sum is interpreted with a smooth function which controls the support being attached. Through out this section we will not display the smooth test functions in each transformations. We will focus on the generic term, ignoring various error terms whose contributions are not essential. Our goal is to beat the convexity bound $O(N)$ of $S(N)$.

We let $P,L\gg 1$ be two parameters such that $PM\gg NL$. By using the Hecke relation $\lambda_g(n\ell)\approx \lambda_g(n)\lambda_g(\ell)$ and the approximation 
$$\delta(n,r\ell)\approx \frac{1}{pM}\sumstar_{\alpha(pM)}e\left(\frac{\alpha(n-r\ell)}{pM}\right),$$
we can write
\begin{equation*}
\begin{split}
S(N)=&\frac{1}{L}\sum_{\ell\sim L}\overline{\lambda_g(\ell)}\sum_{n\sim N\ell}\lambda_g(n)\sum_{r\sim N}\chi(r)\delta(n,r\ell)\\
\approx&\frac{1}{L}\sum_{\ell\sim L}\overline{\lambda_g(\ell)}\sum_{n\sim NL}\lambda_g(n)\sum_{r\sim N}\chi(r)\frac{1}{P}\sum_{p\sim P}\frac{1}{pM}\sumstar_{\alpha(pM)}e\left(\frac{\alpha(n-r\ell)}{pM}\right)\\
\approx& \frac{1}{P^2ML}\sum_{p\sim P}\sum_{\ell\sim L}\overline{\lambda_g(\ell)}\sumstar_{\alpha(pM)}\sum_{n\sim NL}\lambda_g(n)e\left(\frac{\alpha n}{pM}\right)\sum_{r\sim N}\chi(r)e\left(\frac{-\alpha r\ell}{pM}\right).
\end{split}
\end{equation*}
Here $p\sim P$ and $\ell\sim L$ denote primes in the dyadic intervals $[P,2P]$ and $[L,2L]$.

We then dualize the $n$ and $r$ sums using Voronoi summation and Poisson summation, respectively, getting
\begin{equation*}\sum_{n\sim NL}\lambda_g(n)e\left(\frac{\alpha n}{pM}\right)\leftrightarrow \frac{NL}{pM}\sum_{n\sim \frac{P^2M^2}{NL}}\overline{\lambda_g(n)}e\left(\frac{-\bar{\alpha}n}{pM}\right)
\end{equation*}
and
\begin{equation*}
\begin{split}
\sum_{r\sim N}\chi(r)e\left(\frac{-\alpha r\ell}{pM}\right)\leftrightarrow& \frac{N}{pM}\sum_{r<\frac{PM}{N}}\sum_{\beta(pM)}\chi(\beta)e\left(\frac{-\alpha \beta \ell }{pM}\right)e\left(\frac{\beta r}{pM}\right) \\
=&\frac{N}{pM}\sum_{r<\frac{PM}{N}}g_\chi \chi(p)\bar{\chi}(r-\alpha \ell)\cdot p\delta_{r\equiv \alpha \ell \bmod p}.
\end{split}\end{equation*}
By splitting $\sumstar_{\alpha(pM)}$ into $\sumstar_{\alpha(M)}$ and $\sumstar_{\alpha(p)}$ and putting things together, we get
\begin{equation*}
\begin{split}
S(N)=&\frac{1}{P^2ML}\sum_{p\sim P}\sum_{\ell\sim L}\overline{\lambda_g(\ell)}\frac{NL}{pM}\sum_{n\sim \frac{P^2M^2}{NL}}\overline{\lambda_g(n)}\\
&\cdot\frac{N}{pM}\sum_{r<\frac{PM}{N}}\left(\sumstar_{\alpha(M)}e\left(\frac{-\bar{\alpha}n\bar{p}}{M}\right)g_\chi \chi(p)\bar{\chi}(r-\alpha \ell)\right) \cdot \sumstar_{\alpha(p)}\left(e\left(\frac{-\bar{\alpha}n\bar{M}}{p}\right)p\delta_{r\equiv \alpha \ell \bmod p}\right)\\
\approx& \frac{N^2}{P^3M^{5/2}}\sum_{n\sim \frac{P^2M^2}{NL}}\overline{\lambda_g(n)}\sum_{p\sim P}\chi(p)\sum_{\ell\sim L}\overline{\lambda_g(\ell)}\sum_{r<\frac{PM}{N}}e\left(\frac{-\bar{r}n\ell\bar{M}}{p}\right)\sumstar_{\alpha(M)}e\left(\frac{-\bar{\alpha}n\ell\bar{p}}{M}\right)\bar{\chi}(r-\alpha).
\end{split}
\end{equation*}
We then use Cauchy--Schwarz to remove the $\rm GL(2)$ coefficients $\overline{\lambda_g(n)}$ and get
\begin{equation*}
\begin{split}
S(N)\ll& \frac{N^2}{P^3M^{5/2}}\left(\sum_{n\sim \frac{P^2M^2}{NL}}|\overline{\lambda_g(n)}|^2\right)^{1/2}\\
&\left(\sum_{n\sim \frac{P^2M^2}{NL}}\bigg|\sum_{p\sim P}\chi(p)\sum_{\ell\sim L}\overline{\lambda_g(\ell)}\sum_{r<\frac{PM}{N}}e\left(\frac{-\bar{r}n\ell\bar{M}}{p}\right)\sumstar_{\alpha(M)}e\left(\frac{-\bar{\alpha}n\ell\bar{p}}{M}\right)\bar{\chi}(r-\alpha)\bigg|^2\right)^{1/2}.
\end{split}
\end{equation*}
\begin{Remark}\label{ell-sum}
The contribution from the ``diagonal term" $(p_1,\ell_1,r_1,\alpha_1)=(p_2,\ell_2,r_2,\alpha_2)$ is given by
\begin{equation*}
\begin{split}
S_{\text{diag}}\ll& \frac{N^2}{P^3M^{5/2}}\frac{P^2M^2}{NL}\left(PL\frac{PM}{N}M\right)^{1/2}\ll \frac{N^{1/2}M^{1/2}}{L^{1/2}},
\end{split}
\end{equation*}
which improves over the trivial bound $O(M)$ as long as $L$ has some size.
\end{Remark}
Opening the absolute valued square and switching the order of summations, we get
\begin{equation*}
\begin{split}
S(N)\ll& \frac{N^{3/2}}{P^2M^{3/2}L^{1/2}}\bigg(\sum_{p_1,p_2}\chi(p_1)\bar{\chi}(p_2)\sum_{\ell_1,\ell_2}\overline{\lambda_g(\ell_1)}\lambda_g(\ell_2)\sum_{r_1<\frac{PM}{N}}\sum_{r_2<\frac{PM}{N}}\sumstar_{\alpha_1(M)}\sumstar_{\alpha_2(M)}\bar{\chi}(r_1-\alpha_1)\chi(r_2-\alpha_2)\\
&\sum_{n\sim \frac{P^2M^2}{NL}}e\left(\frac{-\overline{r_1}n\ell_1\bar{M}}{p_1}+\frac{\overline{r_2}n\ell_2\bar{M}}{p_2}\right)e\left(\frac{-\overline{\alpha_1}n\ell_1\overline{p_1}}{M}+\frac{\overline{\alpha_2}n\ell_2\overline{p_2}}{M}\right)\bigg)^{1/2}.
\end{split}
\end{equation*}
The $n$-sum in the second line above, by applying Poisson summation, is
\begin{equation*}
\begin{split}
\sum_{n\sim \frac{P^2M^2}{NL}}(\cdots)\leftrightarrow& \frac{P^2M^2}{NL}\frac{1}{p_1p_2M}\sum_{n\ll \frac{p_1p_2M}{P^2M^2/NL}}\sum_{\beta(p_1p_2)}e\left(\frac{\beta(-\overline{r_1}\ell_1p_2+\overline{r_2}\ell_2p_1+n)\bar{M}}{p_1p_2}\right)\\
&\quad\quad\sum_{\gamma(M)}e\left(\frac{\gamma(-\overline{\alpha_1}\ell_1\overline{p_1}+\overline{\alpha_2}\ell_2\overline{p_2}+n\overline{p_1p_2})}{M}\right)\\
=&\frac{P^2M^2}{NL}\sum_{n\ll \frac{NL}{M}}\delta_{-\overline{r_1}\ell_1p_2+\overline{r_2}\ell_2p_1+n\equiv 0\bmod p_1p_2}\cdot \delta_{\alpha_2\equiv \ell_2p_1\overline{\overline{\alpha_1}\ell_1p_2-n}\bmod M}.
\end{split}
\end{equation*}
Plugging this back to the previous estimate, we obtain
\begin{equation}\label{averaged-sum}
\begin{split}
S(N)\ll& \frac{N}{PM^{1/2}L}\bigg|\sum_{p_1,p_2}\chi(p_1)\bar{\chi}(p_2)\sum_{\ell_1,\ell_2}\overline{\lambda_g(\ell_1)}\lambda_g(\ell_2)\sum_{r_1,r_2<\frac{PM}{N}}\sum_{n\ll \frac{NL}{M}}\delta_{-\overline{r_1}\ell_1p_2+\overline{r_2}\ell_2p_1+n\equiv 0\bmod p_1p_2}\\
&\cdot \sumstar_{\alpha(M)}\bar{\chi}(r_1-\alpha)\chi\left(r_2-\ell_2p_1\overline{\bar{\alpha}\ell_1p_2-n}\right)\bigg|^{1/2}.
\end{split}
\end{equation}
The contribution from the zero-frequency $n=0$ will essentially correspond to the diagonal contribution $$S_{\text{diag}}\ll \frac{N^{1/2}M^{1/2}}{L^{1/2}}$$ in Remark \ref{ell-sum}.

For $n\neq 0$, assuming the parameters $(p_1,\ell_1,r_1,p_2,\ell_2,r_2)$ are in the generic position, then by appealing to Weil's theory, one has
\begin{equation}\label{Weil' theory}
\frac{1}{M^{1/2}}\sumstar_{\alpha(M)}\bar{\chi}(r_1-\alpha)\chi\left(r_2-\ell_2p_1\overline{\bar{\alpha}\ell_1p_2-n}\right)\ll 1.
\end{equation}
Hence the contribution $S_{\text{off}}$ from the non-zero frequencies $n\neq 0$ is given by
\begin{equation*}
\begin{split}
S_{\text{off}}\ll \frac{N}{PM^{1/2}L}\bigg(P^2L^2\frac{PM}{N}\frac{PM}{N}
\frac{NL}{M}M^{1/2}\frac{1}{P^2}\bigg)^{1/2}\ll N^{1/2}M^{1/4}L^{1/2}.
\end{split}
\end{equation*}
In conclusion, we obtain 
\begin{equation*}
\begin{split}
S(N)\ll S_{\text{diag}}+S_{\text{off}}\ll \frac{N^{1/2}M^{1/2}}{L^{1/2}}+N^{1/2}M^{1/4}L^{1/2}.
\end{split}
\end{equation*}
By choosing $L=M^{1/4}$ and $P>\frac{N}{M^{3/4}}$ we obtain
\begin{equation*}
\begin{split}
S(N)\ll  N^{1/2}M^{3/8},
\end{split}
\end{equation*}
which will imply the Burgess bound as stated in Theorem \ref{cusp form thm}.
\begin{Remark}
One can interpret the left hand side of \eqref{Weil' theory} as a \emph{Frobenius trace function} $K(n)$ (depending on $(p_1,\ell_1,r_1,p_2,\ell_2,r_2)$) modulo $M$. In order to improve the Burgess bound, one might think to improve the bound \eqref{Weil' theory} ``on average", by taking advantage of the extra summation over $n$ in \eqref{averaged-sum}. But this is difficult since the length of the summation is below the Poly\' a--Vinogradov range.
\end{Remark}

\section{The set-up}
For any $N\geq 1$, define the following sum
\begin{equation}\label{S(N)}
S(N)=\sum_{n=1}^{\infty}\lambda_g(n)\chi(n)W\left(\frac{n}{N}\right),
\end{equation}
where $W$ is a smooth bump function supported on $[1,2]$ with $W^{(j)}(x)\ll_j 1$. Estimating the sum trivially with \eqref{Rankin--Selberg} gives the bound $S(N)\ll_\varepsilon N^{1+\varepsilon}$. Using an approximate functional equation of $L(s, g\otimes\chi)$, one can derive the following.
\begin{Lemma}\label{approximate functional equation}
For any $0<\delta<1$, we have
\begin{equation*}
L\left(\frac{1}{2},g\otimes \chi \right)\ll M^{\varepsilon}\sup_{N}\frac{|S(N)|}{\sqrt{N}}+M^{1/2-\delta/2+\varepsilon},
\end{equation*}
where the supremum is taken over $N$ in the range $M^{1-\delta}<N<M^{1+\varepsilon}$.
\end{Lemma}
 From the above lemma, it suffices to improve the bound $S(N)\ll_\varepsilon N^{1+\varepsilon}$ in the range $M^{1-\delta}<N<M^{1+\varepsilon}$, where $\delta>0$ is a constant to be chosen later. 

Let $\mathcal{L}$ be the set of primes $\ell$ in the dyadic interval $[L,2L]$, where $L<M^{1-\varepsilon}$ is a parameter to be determined later. Denote $L^\star=\sum_{\ell\in \mathcal{L}}|\lambda_g(\ell)|^2$. Then $L^\star\asymp \frac{L}{\log L}$, as the following argument shows.
For $\varepsilon>0$,
\begin{equation*}
\sum_{\ell\in \mathcal {L}}|\lambda_g(\ell)|^2\asymp \frac{1}{\log L}
\sum_{\ell\in \mathcal {L}}(\log \ell)|\lambda_g(\ell)|^2
=\frac{1}{\log L}\sum_{n=L}^{2L}\Lambda(n)|\lambda_g(n)|^2
+O\left(L^{1-\varepsilon}\right),
\end{equation*}
where $\Lambda(n)$ is the Von Mangoldt function. By the prime number theorem for
automorphic representations (see \cite[Corollary 1.2]{LiuWangYe}), we have $\sum_{L\leq n\leq 2L}\Lambda(n)|\lambda_g(n)|^2 \sim L$. 
Thus
\begin{equation*}
L^\star=\sum_{\ell\in \mathcal {L}}|\lambda_g(\ell)|^2\asymp \frac{L}{\log L}.
\end{equation*}

Similarly, we let $P$ be a parameter and $\mathcal{P}$ be the set of primes $p$ in the dyadic interval $[P,2P]$. Denote $P^\star=\sum_{p\in \mathcal{P}}1\asymp \frac{P}{\log P}$. We will choose $P$ and $L$ so that $\mathcal{P}\cap \mathcal{L}=\emptyset$.

Let $p\in \mathcal{P}$, $n\asymp NL$ and $r\asymp N$. For $\varepsilon>0$ and $PM\gg (NL)^{1+\varepsilon}$, the condition $n=r\ell$ is equivalent to the congruence $n\equiv r\ell\bmod pM$. Since $N<M^{1+\varepsilon}$, we assume that, 
\begin{equation}\label{restriction 1}
P\gg L^{1+\varepsilon}.
\end{equation}
Therefore, under the assumption
\begin{equation*}\label{restriction 2}
PM\gg (NL)^{1+\varepsilon},
\end{equation*}
by using the detection \eqref{trivial delta} with $q=pM$, the main sum of interest $S(N)$ defined in \eqref{S(N)} can be expressed as
\begin{equation}\label{start}
\begin{split}
S(N)
=&\frac{1}{L^\star}\sum_{\ell\in\mathcal{L}}\overline{\lambda_g(\ell)}\sum_{n=1}^{\infty}\lambda_g(n)W\left(\frac{n}{N\ell}\right)\sum_{r=1}^{\infty}\chi(r)V\left(\frac{r}{N}\right)\delta(n=r\ell)\left(\frac{n}{r\ell}\right)^{iv}+O\left(\frac{N^{1+\varepsilon}}{L}\right)\\
=&\frac{1}{L^\star P^\star}\sum_{\ell\in\mathcal{L}}\overline{\lambda_g(\ell)}\sum_{p\in\mathcal{P}}\frac{1}{pM}\sum_{c|pM}\quad\sumstar_{\alpha\bmod c}\sum_{n=1}^{\infty}\lambda_g(n)e\left(\frac{\alpha n}{c}\right)\left(\frac{n}{N\ell}\right)^{iv}W\left(\frac{n}{N\ell}\right)\\&\qquad\qquad\qquad\qquad\times\sum_{r=1}^{\infty}\chi(r)e\left(\frac{-\alpha r\ell}{c}\right)\left(\frac{r}{N}\right)^{-iv}V\left(\frac{r}{N}\right) + O\left(\frac{N^{1+\varepsilon}}{L}\right).
\end{split}
\end{equation}
Here $V$ is a smooth function supported on $[1/2,3]$, constantly $1$ on $[1,2]$ and satisfies $V^{(j)}(x)\ll_j 1$, and $\overline{\lambda_g(\ell)}$ denotes $\tau(\ell)^{-1}$ when $\lambda_g$ is the divisor function $\tau$, and
\begin{equation}\label{size-of-v}
    v:=M^{\varepsilon}.
\end{equation}The error term $O\left(N^{1+\varepsilon}L^{-1}\right)$ arises from the Hecke relation $\lambda_g(r\ell)=\lambda_g(r)\lambda_g(\ell)-\delta_{\ell |r}\lambda_g(1)\lambda_g(r/\ell)$. 

\begin{Remark}(1) The point of introducing the extra $(n/r\ell)^{iv}$ factor with $v=M^{\varepsilon}$ is to insure that, after applying Voronoi and Poisson summations, in \eqref{after dual summations} the dual $n$ and $r$ summations are essentially supported on dyadic intervals, which would help simplify the subsequent counting arguments.
\par
(2) The extra sum over $\ell$ is reminiscent of the amplification technique in \cite{DFI} (see also \cite{Munshi17, HN17, Lin}). Without introducing it we will be at the threshold to beat the convexity bound (cf. Remark \ref{ell-sum}).
\end{Remark}
Our strategy then is to apply dual summation formulas to the $n$ and $r$-sums, followed by applications of Cauchy--Schwarz inequality and Poisson summation to the $n$-sum. A careful analysis of the resulting congruence conditions together with Deligne's theory of exponential sums yields the final bounds.

\section{Application of dual summation formulas}\label{Voronoi and Poisson}

We start with an application of the Voronoi summation formula (Lemma \ref{voronoi}) to the $n$-sum in equation \eqref{start}. Then,
\begin{equation}\label{after voronoi1}
\begin{split}
S(N)
=&\mathcal{I}_{g=\tau}+\frac{N}{L^\star P^\star}\sum_{\pm}\sum_{\ell\in\mathcal{L}}\overline{\lambda_g(\ell)}\ell\sum_{p\in\mathcal{P}}\frac{1}{pM}\sum_{c|pM}\frac{1}{c}\quad\sum_{n=1}^{\infty}\lambda_g(n)\widecheck{W}_{v,g}^{\pm}\left(\frac{n}{c^2/N\ell}\right)\\&\qquad\qquad\qquad\qquad\times\sum_{r=1}^{\infty}\chi(r)S(r\ell,\pm n;c)\left(\frac{r}{N}\right)^{-iv}V\left(\frac{r}{N}\right)+ O\left(\frac{N^{1+\varepsilon}}{L}\right),
\end{split}
\end{equation}
where $\mathcal{I}_{g=\tau}$ vanishes if $g$ is a cusp form, otherwise it is given by
\begin{equation*}
\begin{split}
\mathcal{I}_{g=\tau}=&\frac{N}{L^\star P^\star}\sum_{\ell\in\mathcal{L}}\overline{\lambda_g(\ell)}\ell\sum_{p\in\mathcal{P}}\frac{1}{pM}\sum_{c|pM}\frac{1}{c}\quad \sumstar_{\alpha\bmod c}\sum_{r=1}^{\infty}\chi(r)e\left(\frac{-\alpha r\ell}{c}\right)\left(\frac{r}{N}\right)^{-iv}V\left(\frac{r}{N}\right)\\
&\quad\quad\times \int_0^\infty \left(\log x+2\gamma+\log N\ell-2\log c\right)W_v(x)\mathrm{d}x. 
\end{split}
\end{equation*}
Here $\widecheck{W}_{v,g}^{\pm}$ denotes the Hankel transform of the function $W_v(y):=y^{iv}W(y)$, defined in Lemma \ref{voronoi}. By performing a stationary phase argument, the function $\widecheck{W}_{v,g}^{\pm}(x)$ is negligibly small, unless $x\asymp v^2=M^{2\varepsilon}$ (for a proof see \cite[Lemma 3.3]{LMS}). Hence the $n$-variable in \eqref{after voronoi1} is supported on $n\asymp \frac{c^2v^2}{N\ell}$.

Next, we apply Poisson summation to the $r$-sums in \eqref{after voronoi1}. We introduce a few notations. For $a,b\in\BZ$, we let $[a,b]$ be the lcm of $a$ and $b$; let $(a,b)$ be the gcd of $a$ and $b$, and let $a_b=a/(a,b)$. We note that if $a$ is squarefree, then $(b,a_b)=1$. 

Writing $S(r\ell,\pm n;c)=\sumstar_{\alpha\bmod c}e\left(\frac{\mp\overline{\alpha} n-\alpha r\ell}{c}\right)$, the $r$-sums in \eqref{after voronoi1} are given by
\begin{equation*}
\sum_{r\geq1}\chi(r)e\left(\frac{-\alpha r\ell}{c}\right)\left(\frac{r}{N}\right)^{-iv}V\left(\frac{r}{N}\right).
\end{equation*}
Breaking the sum modulo $[c,M]$ and applying the Poisson summation formula, the $r$-sum becomes
\begin{equation}\label{r-sum}
\frac{N}{[c,M]} \sum_{r\in\BZ}\left(\sum_{\beta\bmod[c,M]}\chi(\beta)e\left(\frac{-\alpha\beta\ell}{c}\right)e\left(\frac{r\beta}{[c,M]}\right)\right)\widehat{V_v}\left(\frac{rN}{[c,M]}\right),
\end{equation}
where $V_v(y):=y^{-iv}V(y)$ and $\widehat{V_v}$ denotes the Fourier transform of $V_v$.

Using the relation $[c,M]=Mc_M$ and reciprocity, the $\beta$-sum can be rewritten as
\begin{equation*}
\begin{split}
&\sum_{\beta\bmod M}\chi(\beta)e\left(\frac{(r-\alpha\ell M_c)\overline{c_M}\beta}{M}\right) \times \sum_{\beta\bmod c_M}e\left(\frac{(r-\alpha\ell M_c)\overline{M}\beta}{c_M}\right)\\
=& \bar{\chi}((r-\alpha\ell M_c)\overline{c_M})g_\chi\times c_M\delta(r-\alpha\ell M_c\equiv0\bmod c_M),
\end{split}
\end{equation*}
where $g_\chi$ is the Gauss sum. By \eqref{Vcheck}, one can truncate the $r$-sum at $|r|\ll [c,M]N^{\varepsilon}/N$, up to a negligible error. Therefore \eqref{r-sum} becomes
\begin{equation*}
\frac{Ng_\chi}{M}\underset{\begin{subarray}{c}
|r|\ll [c,M]N^{\varepsilon}/N \\ r-\alpha\ell M_c\equiv0\bmod c_M
\end{subarray}}{\sum}\bar{\chi}((r-\alpha\ell M_c)\overline{c_M}) \widehat{V_v}\left(\frac{rN}{[c,M]}\right) + O(N^{-2018}).
\end{equation*}
Substituting the above expression into \eqref{after voronoi1}, we arrive at
\begin{equation}\label{after dual summations}
\begin{split}
S(N)=\mathcal{M}_{g=\tau}+\frac{N^2g_\chi}{ML^\star P^\star}&\sum_{\pm}\sum_{\ell\in\mathcal{L}}\overline{\lambda_g(\ell)}\ell\sum_{p\in\mathcal{P}}\frac{1}{pM}\sum_{c|pM}\frac{1}{c}\quad\sumstar_{\alpha\bmod c}\sum_{n=1}^{\infty}\lambda_g(n)e\left(\mp\frac{\overline{\alpha} n}{c}\right)\widecheck{W}_{v,g}^{\pm}\left(\frac{n}{c^2/N\ell}\right)\\&\underset{\begin{subarray}{c}
|r|\ll [c,M]N^{\varepsilon}/N \\ r-\alpha\ell M_c\equiv0\bmod c_M
\end{subarray}}{\sum}\bar{\chi}((r-\alpha\ell M_c)\overline{c_M}) \widehat{V_v}\left(\frac{rN}{[c,M]}\right) + O\left(\frac{N^{1+\varepsilon}}{L}\right).
\end{split}\end{equation}
The term $\mathcal{M}_{g=\tau}$ vanishes if $g$ is a cusp form. Otherwise it is given by
\begin{equation}\label{M(g=tau) bound}
\begin{split}
\mathcal{M}_{g=\tau}=&\frac{N^2g_\chi}{ML^\star P^\star}\sum_{\ell\in\mathcal{L}}\overline{\lambda_g(\ell)}\ell\sum_{p\in\mathcal{P}}\frac{1}{pM}\sum_{c|pM}\frac{1}{c}\quad\sumstar_{\alpha\bmod c} \quad\underset{\begin{subarray}{c}
|r|\ll [c,M]N^{\varepsilon}/N \\ r-\alpha\ell M_c\equiv0\bmod c_M
\end{subarray}}{\sum}\bar{\chi}((r-\alpha\ell M_c)\overline{c_M}) \widehat{V_v}\left(\frac{rN}{[c,M]}\right)\\
&\quad\times\int_0^\infty \left(\log x+2\gamma+\log N\ell-2\log c\right)W_v(x)\mathrm{d}x\\
\ll & \frac{N(PM)^\varepsilon}{\sqrt{M}LP}\sum_{1\leq\ell\leq 2L}|\lambda_g(\ell)|\ell\\
\ll& (PML)^{\varepsilon}\frac{NL}{PM^{1/2}}.
\end{split}
\end{equation}
The last inequality is deduced by an application of Cauchy--Schwarz inequality and lemma \ref{Ramanujan}. We similarly bound the sums in \eqref{after dual summations} corresponding to $c=1, p, M$. When $c=1$, we get arbitrarily small contribution because of the weight functions $\widecheck{W}^\pm$. When $c=p$, we again get arbitrarily small contribution (because of the weight functions $\widecheck{W}^\pm$) since we will choose $P$ such that 
\begin{equation}\label{c=p}
P^2<M^{1-\delta}L.
\end{equation}
When $c=M$, we have
\begin{equation}\label{c=M}
\begin{split}
&\frac{N^2g_\chi}{ML^\star P^\star}\sum_{\pm}\sum_{\ell\in\mathcal{L}}\overline{\lambda_g(\ell)}\ell\sum_{p\in\mathcal{P}}\frac{1}{pM^2}\sum_{n=1}^{\infty}\lambda_g(n)\widecheck{W}_{v,g}^{\pm}\left(\frac{n}{M^2/N\ell}\right)\\&\times\underset{\begin{subarray}{c}
|r|\ll MN^{\varepsilon}/N\end{subarray}}{\sum}\left(\underset{\alpha\bmod M}{\sumstar} \bar{\chi}(r-\alpha\ell)e\left(\mp\frac{\overline{\alpha} n}{M}\right)\right)\widehat{V_v}\left(\frac{rN}{M}\right)\\
\ll & (PML)^{\varepsilon}\frac{M}{P}.
\end{split}
\end{equation}
This bound is obtained by making use of Lemma \ref{squareroot 1} to get $\underset{\alpha\bmod M}{\sumstar} \bar{\chi}(r-\alpha\ell)e\left(\mp\frac{\overline{\alpha} n}{M}\right)\ll M^{1/2}$. Therefore from \eqref{after dual summations}, \eqref{M(g=tau) bound}  and \eqref{c=M}, we obtain,

\begin{equation}\label{S(N) and S(N) star}
S(N) = S^\star(N) + O\left((PML)^{\varepsilon}\left(\frac{NL}{PM^{1/2}} + \frac{M}{P} + \frac{N}{L}\right)\right)
\end{equation}
under the condition $P^2<M^{1-\delta}L$, where
\begin{equation*}
\begin{split}
S^\star(N) = &\frac{N^2g_\chi}{M^3L^\star P^\star}\sum_{\pm}\sum_{\ell\in\mathcal{L}}\overline{\lambda_g(\ell)}\ell\sum_{p\in\mathcal{P}}\frac{\chi(p)}{p^2}\sum_{n=1}^{\infty} \lambda_g(n)\widecheck{W}_{v,g}^{\pm}\left(\frac{n}{p^2M^2/N\ell}\right)\\&\times\underset{\begin{subarray}{c}
|r|\ll pMN^{\varepsilon}/N 
\end{subarray}}{\sum}\left(\underset{\alpha\bmod pM}{\sumstar}\bar{\chi}(r-\alpha\ell)e\left(\mp\frac{\overline{\alpha} n}{pM}\right)\delta(r-\alpha\ell \equiv0\bmod p)\right) \widehat{V_v}\left(\frac{rN}{pM}\right) + O(N^{-2018}).
\end{split}
\end{equation*}

Since $(p,M)=1$, the sum $\underset{\alpha\bmod pM}{\sumstar}\bar{\chi}(r-\alpha\ell)e\left(\mp\frac{\overline{\alpha} n}{pM}\right)\delta(r-\alpha\ell \equiv0\bmod p)$ factors as
\begin{equation*}
\begin{split}
&\underset{\alpha\bmod M}{\sumstar}\bar{\chi}(r-\alpha\ell)e\left(\mp\frac{\overline{\alpha p} n}{M}\right) \times \underset{\alpha\bmod p}{\sumstar}e\left(\mp\frac{\overline{\alpha M} n}{p}\right)\delta(r-\alpha\ell \equiv0\bmod p)\\
&=e\left(\mp\frac{\overline{rM} n\ell}{p}\right) \underset{\alpha\bmod M}{\sumstar}\bar{\chi}(r-\alpha\ell)e\left(\mp\frac{\overline{\alpha p} n}{M}\right),
\end{split}\end{equation*}
Since $(r,p)=1$, we have $r\neq0$. Moreover, it suffices to estimate the `minus' term of $S^\star(N)$ since the estimates of the `plus' terms will be similar. By abuse of notation, we write $\widecheck{W}_{v,g}^{-}$ as $\widecheck{W}_v$. Then,
\begin{equation*}
\begin{split}
S^\star(N)=&\frac{N^2g_{\chi}}{L^\star P^\star M^3}\sum_{n=1}^{\infty}\lambda_g(n)\sum_{\ell\in\mathcal{L}}\overline{\lambda_g(\ell)}\ell\sum_{p\in\mathcal{P}}\frac{\chi(p)}{p^2}\sum_{(r,p)=1}e\left(\frac{-\bar{r}n\ell \bar{M}}{p}\right)\\
&\sumstar_{\alpha(M)}\bar{\chi}(r+\alpha)e\left(\frac{\bar{\alpha}n\ell\bar{p}}{M}\right)\widehat{V_v}\left(\frac{r}{pM/N}\right)\widecheck{W}_v\left(\frac{n}{p^2M^2/N\ell}\right).
\end{split}
\end{equation*}

Munshi treated a sum similar to ours in \cite[P.13]{Munshi17}. For the sake of completeness, we will carry out the details, but our arguments closely follow that of Munshi \cite[Section 7]{Munshi17}.

 \begin{Remark}
 (1) At this stage, if we estimate the sum $S^\star(N)$ directly (with the help of Lemma \ref{squareroot 1}), we would get
 \begin{equation*}
\begin{split}
S^\star(N)\ll N^{\varepsilon}\frac{N^2M^{1/2}}{LP M^3}\frac{P^2M^2}{NL}L^2\frac{1}{P}\frac{PM}{N}M^{1/2}\ll N^{\varepsilon}PM.
\end{split}
\end{equation*}
This falls short of $O(PM^{\eta})$ from the target bound $O(M^{1-\eta})$. In the sequel, we will use Cauchy--Schwarz to smooth the $n$-variable and apply Poisson summation thereafter to obtain extra saving.
  \par
  (2) It is natural to try to compare our method to that used by \cite{Bykovskii}, and by \cite{Blomer-Harcos, FKM}.   The ingredients are similar, the results are the same, and both methods work by embedding the modular form $g$ in question into some larger space of cusp forms on $\Gamma_0(M)$ and averaging, in our case indirectly via the Cauchy--Schwarz inequality, in \cite{Bykovskii, Blomer-Harcos, FKM}'s case more directly via the Petersson formula. We thank Paul Nelson for pointing out this analogy.
 \end{Remark}

\section{Cauchy--Schwarz and Poisson Summation}\label{Cauchy--Schwarz and Poisson}

In the following, whenever we need to bound the Fourier coefficients $\lambda_g(n)$ for $g$ a Maass form, we simply apply the Rankin--Selberg estimate $\sum_{n\leq X}\left|\lambda_g(n)\right|^2\ll X$, as a substitute of the Ramanujan bound for individual coefficients. Recall from the test function $\widecheck{W}_v$ we know that the $n$-sum is supported on $n\sim \frac{p^2M^2v^2}{N\ell}$. We set
\begin{equation*}
\CN_0=\frac{p^2M^2v^2}{NL}.
\end{equation*}
Up to an arbitrarily small error, we arrive at
\begin{equation*}
\begin{split}
S^\star(N)\ll \frac{N^{2+\varepsilon}}{L^\star P^\star M^{5/2}}&\sum_{n}|\lambda_g(n)|U\left(\frac{n}{\CN_0}\right)\bigg|\sum_{\ell\in\mathcal{L}}\overline{\lambda_g(\ell)}\ell\sum_{p\in\mathcal{P}}\frac{\chi(p)}{p^2}\sum_{(r,p)=1}e\left(\frac{-\bar{r}n\ell \bar{M}}{p}\right)\\
&\sumstar_{\alpha(M)}\bar{\chi}(r+\alpha)e\left(\frac{\bar{\alpha}n\ell\bar{p}}{M}\right)\widehat{V_v}\left(\frac{r}{pM/N}\right) \widecheck{W}_v\left(\frac{n}{p^2M^2/N\ell}\right)\bigg|.
\end{split}
\end{equation*}
Here $U$ is a smooth function with compact support contained in $\mathbb{R}_{>0}$.
Applying the Cauchy--Schwarz inequality to the $n$-sum and using the Ramanujan bound on average,
\begin{equation}\label{S star before}
\begin{split}
S^\star(N)
\ll& (NML)^{\varepsilon}\frac{N^2}{LPM^{5/2}}\CN_0^{1/2}S^\star(N,\CN_0)^{1/2} + N^{-2018},
\end{split}
\end{equation}
where
\begin{equation*}
\begin{split}
S^\star(N,\CN_0)=\sum_{n} U\left(\frac{n}{\CN_0}\right)\bigg|\sum_{\ell\in\mathcal{L}}\overline{\lambda_g(\ell)}\ell & \sum_{p\in\mathcal{P}}\frac{\chi(p)}{p^2}\sum_{(r,p)=1}e\left(\frac{-\bar{r}n\ell \bar{M}}{p}\right)\\ &\times \sumstar_{\alpha(M)}\bar{\chi}(r+\alpha)e\left(\frac{\bar{\alpha}n\ell\bar{p}}{M}\right)\widehat{V_v}\left(\frac{r}{pM/N}\right) \widecheck{W}_v\left(\frac{n}{p^2M^2/N\ell}\right)\bigg|^2.
\end{split}
\end{equation*}
Opening the square above and switching the order of summations, it suffices to bound the following 
\begin{equation}\label{before truncate}
\begin{split}
S^\star(N, \mathcal{N}_0)
&=
\sum_{\ell_1\in\mathcal{L}}\overline{\lambda_g(\ell_1)}\ell_1\sum_{\ell_2\in\mathcal{L}}\lambda_g(\ell_2)\ell_2\sum_{p_1\in\mathcal{P}}\sum_{p_2\in\mathcal{P}}\frac{\chi(p_1)\bar{\chi}(p_2)}{(p_1p_2)^2}\\
&\sum_{(r_1,p_1)=1}\sum_{(r_2,p_2)=1}\widehat{V_v}\left(\frac{r_1}{p_1M/N}\right)\overline{\widehat{V_v}\left(\frac{r_2}{p_2M/N}\right)}
\sumstar_{\alpha_1(M)}\bar{\chi}(r_1+\alpha_1)
\sumstar_{\alpha_2(M)}\chi(r_2+\alpha_2)\times\mathbf{T},
\end{split}
\end{equation}
with
\begin{equation*}
\begin{split}
\mathbf{T}=\sum_{n=1}^{\infty}e\left(\frac{-\overline{r_1}n\ell_1 \bar{M}}{p_1}\right)&e\left(\frac{\overline{r_2}n\ell_2 \bar{M}}{p_2}\right)e\left(\frac{\overline{\alpha_1}n\ell_1\overline{p_1}-\overline{\alpha_2}n\ell_2\overline{p_2}}{M}\right)\\
&\times U\left(\frac{n}{\mathcal{N}_0}\right)\widecheck{W}_v\left(\frac{n}{p_1^2M^2/N\ell_1}\right)\overline{\widecheck{W}_v\left(\frac{n}{p_2^2M^2/N\ell_2}\right)}.
\end{split}\end{equation*}

We note that one can truncate the $r_1$, $r_2$-sums in \eqref{before truncate} at $|r_1|,|r_2|\ll N^{\varepsilon}\frac{PM}{N}$, at the cost of a negligible error. For smaller values of $r_1$ and $r_2$, we will use the trivial bounds $\widehat{V_v}\left(\frac{r_1}{p_1M/N}\right), \widehat{V_v}\left(\frac{r_2}{p_2M/N}\right)\ll 1$.

Breaking the above $n$-sum modulo $p_1p_2M$ and applying Poisson summation to it,
\begin{equation*}\label{n sum after poisson}
\begin{split}
\mathbf{T}=&\frac{\mathcal{N}_0}{p_1p_2M}\sum_{n}\sum_{b\bmod{p_1p_2M}}e\left(\frac{-\overline{r_1}b\ell_1 \overline{M}}{p_1}\right) e\left(\frac{\overline{r_2}b\ell_2 \overline{M}}{p_2}\right)\\&e\left(\frac{\overline{\alpha_1}b\ell_1\overline{p_1}-\overline{\alpha_2}b\ell_2\overline{p_2}}{M}\right)e\left(\frac{bn}{p_1p_2M}\right)\, \CJ\left(\frac{n}{p_1p_2M/\mathcal{N}_0}\right),
\end{split}
\end{equation*}
where
\begin{equation*}\label{the integral J}
\begin{split}
\mathcal{J}\left(\frac{n}{p_1p_2M/\mathcal{N}_0}\right):=\int_{\BR} U(x)e\left(-\frac{n\mathcal{N}_0x}{p_1p_2M}\right)\widecheck{W}_v\left(\frac{x\CN_0}{p_1^2M^2/N\ell_1}\right)\overline{\widecheck{W}_v\left(\frac{x\CN_0}{p_2^2M^2/N\ell_2}\right)}\mathrm{d}x.
\end{split}
\end{equation*}
The integral $\CJ\left(\frac{n}{p_1p_2M/\mathcal{N}_0}\right)$ gives arbitrarily power saving in $N$ if $|n|\gg N^{\varepsilon}\frac{p_1p_2M}{\mathcal{N}_0}$. Hence we can truncate the dual $n$-sum at 
\begin{equation}\label{new-n-sum}
|n|\ll N^{\varepsilon}\frac{NL}{Mv^2},
\end{equation}at the cost of a negligible error. For smaller values of $n$, we use the trivial bound $\CJ\left(\frac{n}{p_1p_2M/\mathcal{N}_0}\right)\ll 1$. Since $(p_1p_2,M)=1$, we apply reciprocity to write
\begin{equation*}
\begin{split}
\mathbf{T}=&\frac{\mathcal{N}_0}{p_1p_2M}\sum_{n}\sum_{b\bmod{p_1p_2}}e\left(\frac{(-\overline{r_1}\ell_1p_2 +\overline{r_2}\ell_2p_1+n) \overline{M}b}{p_1p_2}\right)\\&\times \sum_{b\bmod M} e\left(\frac{(\overline{\alpha_1}\ell_1p_2-\overline{\alpha_2} \ell_2p_1+n)\overline{p_1p_2}b}{M}\right)\, \CJ\left(\frac{n}{p_1p_2M/\mathcal{N}_0}\right)\\
=&\mathcal{N}_0\underset{\begin{subarray}{c}
|n|\ll N^{\varepsilon}\frac{p_1p_2M}{\CN_0} \\ -\overline{r_1}\ell_1p_2 +\overline{r_2}\ell_2p_1+n\equiv0\bmod p_1p_2\\ (\overline{\alpha_1}\ell_1p_2+n,M)=1 \end{subarray}}{\sum} \delta(\alpha_2\equiv \ell_2p_1(\overline{\overline{\alpha_1}\ell_1p_2+n})\bmod{M})\, \CJ\left(\frac{n}{p_1p_2M/\mathcal{N}_0}\right)+O(N^{-20180}).
\end{split}
\end{equation*}
We must clarify that $\overline{r_i}$ is the inverse of $r_i\bmod p_i$ (and not $\bmod \,p_1p_2$). Substituting the above into \eqref{before truncate},
\begin{equation*}
\begin{split}
S^\star(N, \mathcal{N}_0)\ll &\CN_0\sum_{\ell_1\in\mathcal{L}} \left|\overline{\lambda_g(\ell_1)}\ell_1\right|\sum_{\ell_2\in\mathcal{L}}\left|\lambda_g(\ell_2)\ell_2\right|\sum_{p_1\in\mathcal{P}}\sum_{p_2\in\mathcal{P}}\frac{1}{(p_1p_2)^2}\\
&\underset{\begin{subarray}{c} 0\neq |r_1|\ll R \\ (r_1,p_1)=1\end{subarray}}{\sum}\quad
\underset{\begin{subarray}{c} 0\neq |r_2|\ll R \\ (r_2,p_2)=1\end{subarray}}{\sum}
\underset{\begin{subarray}{c}
|n|\ll N^{\varepsilon}\frac{p_1p_2M}{\CN_0} \\ -\overline{r_1}\ell_1p_2 +\overline{r_2}\ell_2p_1+n\equiv0\bmod p_1p_2 \end{subarray}}{\sum} \left|\mathfrak{C}\right| + N^{-2018},
\end{split}
\end{equation*}
where
\begin{equation}\label{character sum}
\mathfrak{C}=\mathop{\sumstar_{\alpha\bmod M}}_{(\bar{\alpha}\ell_1p_2+n,M)=1}\bar{\chi}(r_1+\alpha)
\chi\bigg(r_2+\ell_2p_1(\overline{\overline{\alpha}\ell_1p_2+n})\bigg),
\end{equation}
and
\begin{equation*}
\begin{split}
R:=N^{\varepsilon}PM/N.
\end{split}
\end{equation*}

When $\ell_1\neq\ell_2$, we apply the Cauchy--Schwarz inequality to the $\ell_i$-sums to get rid of the Fourier coefficients $\lambda_g(\ell_i)$ by using the Ramanujan bound on average. Then,
\begin{equation*}
S^\star(N,\CN_0) \ll S_0^\star(N,\CN_0) + S_1^\star(N,\CN_0) + N^{-2018}
\end{equation*}
where
\begin{equation*}
S_0^\star(N,\CN_0) = \CN_0 \sum_{\ell\in\CL} |\lambda_g(\ell)|^2 \ell^2 \sum_{p_1\in\mathcal{P}}\sum_{p_2\in\mathcal{P}}\frac{1}{(p_1p_2)^2}\underset{\begin{subarray}{c} 0\neq |r_1|\ll R \\ (r_1,p_1)=1\end{subarray}}{\sum}\quad
\underset{\begin{subarray}{c} 0\neq |r_2|\ll R \\ (r_2,p_2)=1\end{subarray}}{\sum}
\underset{\begin{subarray}{c}
|n|\ll N^{\varepsilon}\frac{p_1p_2M}{\CN_0} \\ -\overline{r_1}\ell p_2 +\overline{r_2}\ell p_1+n\equiv0\bmod p_1p_2 \end{subarray}}{\sum} \left|\mathfrak{C}\right|,
\end{equation*}
and
\begin{equation*}
\begin{split}
S_1^\star(N, \mathcal{N}_0)= \CN_0 L^{1+\varepsilon}\bigg(\mathop{\sum_{\ell_1\in\CL} \sum_{\ell_2\in\CL}}_{\ell_1\neq \ell_2}\ell_1^2\ell_2^2 \bigg[&\sum_{p_1\in\mathcal{P}}\sum_{p_2\in\mathcal{P}}\frac{1}{(p_1p_2)^2}\\
&\times\underset{\begin{subarray}{c} 0\neq |r_1|\ll R \\ (r_1,p_1)=1\end{subarray}}{\sum}\quad
\underset{\begin{subarray}{c} 0\neq |r_2|\ll R \\ (r_2,p_2)=1\end{subarray}}{\sum}
\underset{\begin{subarray}{c}
|n|\ll N^{\varepsilon}\frac{p_1p_2M}{\CN_0} \\ -\overline{r_1}\ell_1p_2 +\overline{r_2}\ell_2p_1+n\equiv0\bmod p_1p_2 \end{subarray}}{\sum} \left|\mathfrak{C}\right|\bigg]^2\bigg)^{1/2}.
\end{split}
\end{equation*}


We will choose $P<M^{1-\delta-\varepsilon}$, so that $R<M$ and therefore $(r_1r_2,M)=1$. The remaining task is to count the number of points satisfying the congruence conditions and bounding the sums.

We divide our analysis into cases and write $S_0^\star(N,\CN_0)\ll \CN_0(\Delta_1 + \Delta_2)$ and $S_1^\star(N,\CN_0)\ll\CN_0 L^{\varepsilon}(\Sigma_1 + \Sigma_2)^{1/2}$. The contribution of the terms with $n\equiv 0\bmod{M}$ is given by $\Delta_1$ and $\Sigma_1$, and the contribution of the terms with $n\nequiv 0\bmod{M}$ is given by $\Delta_2$ and $\Sigma_2$, with $\Delta_i$ and $\Sigma_j$ appropriately defined. 

\subsection{$\underline{n\equiv 0\bmod{M}}$}
For the sum \eqref{character sum}, Lemma \ref{squareroot-2} (with the parameter $(\alpha,\beta)=(\ell_2p_1, \ell_1p_2)$ being applied) shows that 
\begin{equation*}
\begin{split}
\mathfrak{C}=\chi(\ell_2p_1\overline{\ell_2p_1})R_M(r_2-r_1\ell_2p_1\overline{\ell_1p_2})-\chi(r_2\overline{r_1})
= \begin{cases} O(M), & \quad \textit{if }  r_2\ell_1p_2\equiv r_1\ell_2 p_1 \bmod M\\
O(1), & \quad \mathrm{otherwise}.\\
\end{cases}
\end{split}\end{equation*}
According to $ r_2\ell_1p_2\equiv r_1\ell_2 p_1\bmod{M}$ or not, we write
\begin{equation*}
\begin{split}
\Delta_1=\Delta_{10}+\Delta_{11} \quad \text{ and } \quad \Sigma_1=\Sigma_{10}+\Sigma_{11},
\end{split}
\end{equation*}
where
\begin{equation*}
\Delta_{10} := \sum_{\ell\in\CL} |\lambda_g(\ell)|^2 \ell^2 \sum_{p_1\in\mathcal{P}}\sum_{p_2\in\mathcal{P}}\frac{1}{(p_1p_2)^2}\underset{r_2p_2\equiv r_1p_1\bmod M}{\underset{\begin{subarray}{c} 0\neq |r_1|\ll R \\ (r_1,p_1)=1\end{subarray}}{\sum}\quad
\underset{\begin{subarray}{c} 0\neq |r_2|\ll R \\ (r_2,p_2)=1\end{subarray}}{\sum}}
\underset{\begin{subarray}{c}
|n|\ll N^{\varepsilon}\frac{p_1p_2M}{\CN_0} \\ -\overline{r_1}\ell p_2 +\overline{r_2}\ell p_1+n\equiv0\bmod p_1p_2 \\ n\equiv 0\bmod M \end{subarray}}{\sum} M,
\end{equation*}
and
\begin{equation*}
\begin{split}
\Sigma_{10}:=L^2\underset{\ell_1\neq\ell_2}{\sum_{\ell_1\in\CL}\sum_{\ell_2\in\CL}} \ell_1^2 \ell_2^2\bigg[\sum_{p_1\in\mathcal{P}}\sum_{p_2\in\mathcal{P}}\frac{1}{(p_1p_2)^2} \underset{r_2\ell_1p_2\equiv r_1\ell_2 p_1\bmod{M}}{\underset{\begin{subarray}{c} 0\neq |r_1|\ll R \\ (r_1,p_1)=1\end{subarray}}{\sum}\quad
\underset{\begin{subarray}{c} 0\neq |r_2|\ll R \\ (r_2,p_2)=1\end{subarray}}{\sum}}
\quad\underset{\begin{subarray}{c}
|n|\ll N^{\varepsilon}\frac{p_1p_2M}{\CN_0} \\ -\overline{r_1}\ell_1p_2 +\overline{r_2}\ell_2p_1+n\equiv0\bmod p_1p_2\\ n\equiv0\bmod M \end{subarray}}{\sum}M\bigg]^2.
\end{split}
\end{equation*}
$\Delta_{11}$ and $\Sigma_{11}$ are the other pieces with the congruence condition $r_2\ell_1p_2\nequiv r_1\ell_2 p_1\bmod{M}$. Opening the square, we write $\Sigma_{10}$ and $\Sigma_{11}$ as a sum over $\ell_i, n, n', r_i, r_i', p_i, p_i'$ for $i=1, 2$. Then, under the assumption $P^2L<N^{1-\varepsilon}$, Lemmas \ref{B10} and \ref{B11} give
\begin{equation}\label{sigma1 bound}
\begin{split}
\Sigma_{10} \ll (PML)^{\varepsilon}\frac{L^6M^4}{P^4N^2} \quad \textit{ and } \quad \Sigma_{11}\ll (PML)^\varepsilon\frac{L^8M^4}{N^4P^4}\left(1+\frac{P^2}{\CN_0}\right)^2,\\
\Delta_{10} \ll (PML)^{\varepsilon}\frac{L^3M^2}{P^2N} \quad \textit{ and } \quad \Delta_{11}\ll (PML)^\varepsilon\frac{L^3M^2}{N^2P^2}\left(1+\frac{P^2}{\CN_0}\right).
\end{split}
\end{equation}

\subsection{$\underline{n\nequiv 0\bmod{M}}$}
Lemma \ref{squareroot-2} shows that 
\begin{equation*}
\mathfrak{C}= \begin{cases} O(M), & \quad \textit{if }  r_1-\bar{n}\ell_1p_2\equiv r_2+\bar{n}\ell_2p_1\equiv 0\bmod M\\
O(M^{1/2}), & \quad \mathrm{otherwise}.\\
\end{cases}
\end{equation*}
According to $r_1-\bar{n}\ell_1p_2\equiv r_2+\bar{n}\ell_2p_1\equiv 0\bmod M$ or not, we write
\begin{equation*}
\begin{split}
\Delta_2 = \Delta_{20} + \Delta_{21} \quad \text{ and } \quad \Sigma_2=\Sigma_{20}+\Sigma_{21},
\end{split}
\end{equation*}
where
\begin{equation*}
\Delta_{20} := \sum_{\ell\in\CL} |\lambda_g(\ell)|^2 \ell^2 \sum_{p_1\in\mathcal{P}}\sum_{p_2\in\mathcal{P}}\frac{1}{(p_1p_2)^2}{\underset{\begin{subarray}{c} 0\neq |r_1|\ll R \\ (r_1,p_1)=1\end{subarray}}{\sum}\quad
\underset{\begin{subarray}{c} 0\neq |r_2|\ll R \\ (r_2,p_2)=1\end{subarray}}{\sum}}
\underset{\begin{subarray}{c}
|n|\ll N^{\varepsilon}\frac{p_1p_2M}{\CN_0} \\ -\overline{r_1}\ell p_2 +\overline{r_2}\ell p_1+n\equiv0\bmod p_1p_2 \\ n\nequiv 0\bmod M \\ r_1-\bar{n}\ell p_2\equiv r_2+\bar{n}\ell p_1\equiv 0\bmod M \end{subarray}}{\sum} M,
\end{equation*}
and
\begin{equation*}
\begin{split}
\Sigma_{20}:= L^2 \underset{\ell_1\neq\ell_2}{\sum_{\ell_1\in\CL}\sum_{\ell_2\in\CL}} \ell_1^2 \ell_2^2\bigg[\sum_{p_1\in\mathcal{P}}\sum_{p_2\in\mathcal{P}}\frac{1}{(p_1p_2)^2} {\underset{\begin{subarray}{c} 0\neq |r_1|\ll R \\ (r_1,p_1)=1\end{subarray}}{\sum}\quad
\underset{\begin{subarray}{c} 0\neq |r_2|\ll R \\ (r_2,p_2)=1\end{subarray}}{\sum}}
\quad\underset{\begin{subarray}{c}
|n|\ll N^{\varepsilon}\frac{p_1p_2M}{\CN_0} \\ -\overline{r_1}\ell_1p_2 +\overline{r_2}\ell_2p_1+n\equiv0\bmod p_1p_2\\ n\nequiv0\bmod M \\ r_1-\bar{n}\ell_1p_2\equiv r_2+\bar{n}\ell_2p_1\equiv 0\bmod M \end{subarray}}{\sum}M\bigg]^2.
\end{split}
\end{equation*}
$\Delta_{21}$ and $\Sigma_{21}$ are the other pieces. Opening the square, we write $\Sigma_{20}$ and $\Sigma_{21}$ as a sum over $\ell_i, n, n', r_i, r_i', p_i, p_i'$ for $i=1, 2$.  Then, from Lemmas \ref{B20} and \ref{B21}
\begin{equation}\label{sigma2 bound}
\begin{split}
\Sigma_{20} \ll (PML)^\varepsilon\frac{L^6M^4}{P^4N^2} \quad \textit{ and } \quad \Sigma_{21}\ll (PML)^\varepsilon\frac{L^8M^5}{N^4P^4}\left(1+\frac{P^2M}{\CN_0}\right)^2,\\
\Delta_{20} \ll (PML)^\varepsilon\frac{L^3M^2}{P^2N} \quad \textit{ and } \quad \Delta_{21}\ll (PML)^\varepsilon\frac{L^3M^{5/2}}{N^2P^2}\left(1+\frac{P^2M}{\CN_0}\right).
\end{split}
\end{equation}

\subsection{Conclusion}
The bounds \eqref{sigma1 bound} and \eqref{sigma2 bound} imply
\begin{equation*}
\begin{split}
S^\star(N,\CN_0)\ll (PML)^{\varepsilon}\CN_0\bigg[\frac{L^3M^2}{P^2N} + \frac{L^4M^{5/2}}{N^2P^2}\bigg(1+\frac{P^2M}{\CN_0}\bigg)\bigg].
\end{split}
\end{equation*}
Inserting this into \eqref{S star before}, 
\begin{equation*}
S^\star(N)\ll (PML)^{\varepsilon}\bigg[\frac{N^{1/2}M^{1/2}v^2}{L^{1/2}} + M^{3/4} + N^{1/2}L^{1/2}M^{1/4}v\bigg].
\end{equation*}
Recalling that $v=M^{\varepsilon}$ \eqref{size-of-v}, and inserting the above into \eqref{S(N) and S(N) star},
\begin{equation*}
\frac{S(N)}{N^{1/2}}\ll (PML)^{\varepsilon}\bigg[\frac{N^{1/2}L}{PM^{1/2}} + \frac{M}{PN^{1/2}} + \frac{M^{1/2}}{L^{1/2}}  + \frac{M^{3/4}}{N^{1/2}} + L^{1/2}M^{1/4}\bigg].
\end{equation*}
The first term is small. Comparing the rest of the terms and the trivial bound of $N^{1/2}$, the optimal choices of parameters turn out to be $P=M^{1/4+\varepsilon}$ and $L=P^{1-\varepsilon}$. Therefore the conditions in \eqref{restriction 1}, \eqref{c=p} and \eqref{restriction 3} are satisfied. In that case,
\begin{equation*}
\frac{S(N)}{N^{1/2}}\ll M^{\varepsilon}\left(\frac{M^{3/4}}{N^{1/2}} + M^{3/8}\right).
\end{equation*}
It therefore makes sense to take $N>M^{3/4}$. For $N\ll M^{3/4}$, we use the trivial bound of $N^{1/2}$. We obtain
\begin{equation*}
L\left(\frac{1}{2},g\otimes \chi \right)\ll M^{3/8+\varepsilon}.
\end{equation*}
That proves Theorem \ref{cusp form thm} and Theorem \ref{divisor thm}.

\appendix

\section{Shifted character sums and counting lemmas}\label{counting and bounding}

For this section, let $M>3$ be a prime and define
\begin{eqnarray*}
\mathfrak{K}=\sum_{z\in \mathbb{F}_M^{\times}}\bar{\chi}(r+\ell z)
e\left(\frac{n\overline{z}}{M}\right),\qquad (\ell,M)=1,\quad n,r,\ell\in \mathbb{Z}.
\end{eqnarray*}

\begin{Lemma}\label{squareroot 1}
Suppose that $(r,M)=1$. If $M|n$, then
$
\mathfrak{K}=-\bar{\chi}(r).
$
If $M\nmid n$, then
\begin{eqnarray*}
\mathfrak{K}\ll M^{1/2}.
\end{eqnarray*}
\end{Lemma}

\begin{proof}
For $M|n$, trivial.
If $M\nmid  n$, by the Fourier expansion of $\chi$ in
terms of additive characters
\begin{eqnarray}\label{Gauss sum}
\chi(a)=g_{\bar{\chi}}^{-1}\sum_{y\bmod M}\bar{\chi}(y)e\left(\frac{a y}{M}\right),
\end{eqnarray}
we have
\begin{eqnarray*}
\mathfrak{K}=g_{\chi}^{-1}\sum_{y,z\in \mathbb{F}_M^{\times}}\chi(y)
e\left(\frac{ry+\ell yz+n\overline{z}}{M}\right).
\end{eqnarray*}
Then the bound follows from  \cite[Corollary 4.3]{Adolphson-Sperber}.
\end{proof}

We define
\begin{equation*}
\mathfrak{C}=\mathop{\sum_{z\in \mathbb{F}_M^{\times}}}_{ (n+\beta\overline{z},M)=1}\bar{\chi}(r_1+z)
\chi\bigg(r_2+\alpha(\overline{n+\beta\overline{z}})\bigg), \qquad (\alpha\beta,M)=1,\quad n,r_1,r_2,\alpha,\beta\in \mathbb{Z}.
\end{equation*}

\begin{Lemma}\label{squareroot-2}
Suppose that $(r_1r_2,M)=1$. If $M|n$, we have
\begin{eqnarray*}
\mathfrak{C}=\chi(\alpha\overline{\beta})
R_M(r_2-r_1\alpha\overline{\beta})-\chi(r_2\overline{r_1}),
\end{eqnarray*}
where $R_M(a)=\sum_{z\in \mathbb{F}_M^{\times}}e(az/M)$ is the Ramanujan sum.
If $M\nmid  n$ and at least one of $r_1-\overline{n}\beta$ and $r_2+\overline{n}\alpha$ is
nonzero in $\mathbb{F}_M$, then
\begin{equation*}
\mathfrak{C}\ll M^{1/2}.
\end{equation*}
Finally, if $n\neq0$ and $r_1-\overline{n}\beta=r_2+\overline{n}\alpha=0$ in $\BF_M$, then 
\[
 \mathfrak{C}=
 \begin{cases}
 -\chi(nr_2\overline{\beta}) \quad &\text{ if } \chi \text{ is not a quadratic character}, \\
 \chi(\overline{n}r_2\beta)(M-1) \quad &\text{ if } \chi \text{ is a quadratic character}.
 \end{cases}
\]
\end{Lemma}

\begin{proof}
For $M|n$,
\begin{eqnarray*}
\mathfrak{C}
=\sum_{z\in \mathbb{F}_M}\bar{\chi}(r_1+z)
\chi(r_2+\alpha\overline{\beta}z)-\chi(r_2\overline{r_1}).
\end{eqnarray*}
Then the first statement follows from \eqref{Gauss sum}. If $M\nmid  n$, by making change of variables $z\rightarrow \overline{n}\beta z$ and $z+1\rightarrow z$, we have
\begin{eqnarray*}
\mathfrak{C}
=\sum_{z\in \mathbb{F}_M^{\times}}
\bar{\chi}(r_1+\overline{n}\beta (z-1))
\chi(r_2+\overline{n}\alpha(1-\overline{z}))-\chi(r_2\overline{r_1}).
\end{eqnarray*}
Applying \eqref{Gauss sum} again we get
\begin{eqnarray*}
\mathfrak{C}
=g_{\chi}^{-1}g_{\bar{\chi}}^{-1}\sum_{x,y,z\in \mathbb{F}_M^{\times}}\chi(x)\bar{\chi}(y)
e\left(\frac{(r_1-\overline{n}\beta)x+
(r_2+\overline{n}\alpha)y+\overline{n}\beta xz-\overline{n} \alpha y\overline{z}}{M}\right)
-\chi(r_2\overline{r_1}).
\end{eqnarray*}
Consider the Newton polyhedron $\Delta(f)$ of
\begin{eqnarray*}
f(x,y,z)=(r_1-\overline{n}\beta)x+
(r_2+\overline{n}\alpha)y+\overline{n}\beta xz-\overline{n} \alpha yz^{-1}\in
\mathbb{F}_M^{\times}[x,y,z,(x y z)^{-1}].
\end{eqnarray*}
We separate into two cases.
\begin{enumerate}
\item If $r_1-\overline{n}\beta=0$ or $r_2+\overline{n}\alpha=0$ in $\mathbb{F}_M$,
then $\Delta(f)$ is the tetrahedron
in $\mathbb{R}^3$ with vertices $(0,0,0)$, $(0,1,0)$, $(1,0,1)$, $(0,1,-1)$
or $(0,0,0)$, $(1,0,0)$, $(1,0,1)$, $(0,1,-1)$. It is easy to check that
$f$ is nondegenerate with respect to $\Delta(f)$. By \cite{Adolphson-Sperber} (see also \cite{Fu1}), we have
$
\mathfrak{C}\ll  M^{1/2}.
$

\item If both $r_1-\overline{n}\beta$ and $r_2+\overline{n}\alpha$ are nonzero in
$\mathbb{F}_M$, then $\Delta(f)$ is the pentahedron
in $\mathbb{R}^3$ with vertices $(0,0,0)$, $(1,0,0)$, $(0,1,0)$,  $(1,0,1)$ and $(0,1,-1)$.
The only face which fails to meet the criterion
for nondegeneracy is the quadrilateral one with polynomial
$
f_{\sigma}(x,y,z)=(r_1-\overline{n}\beta)x+
(r_2+\overline{n}\alpha)y+\overline{n}\beta xz-\overline{n} \alpha yz^{-1}
$
for which the locus of
$
\partial f_{\sigma}/\partial x=\partial f_{\sigma}/\partial y=
\partial f_{\sigma}/\partial z=0
$
is empty in $(\mathbb{F}_M^{\times})^3$ only if
$n\not\equiv \beta \overline{r_1}-\alpha \overline{r_2}\bmod M$.
Therefore, for $n\not\equiv \beta \overline{r_1}-\alpha \overline{r_2}\bmod M$,
we can apply the square-root cancellation result in \cite{Adolphson-Sperber} or \cite{Fu1} to get
$\mathfrak{C}\ll M^{1/2}$.
\end{enumerate}

For $n\equiv \beta \overline{r_1}-\alpha \overline{r_2}\bmod M$, we have
$
r_1-\overline{n}\beta\equiv -\alpha r_1\overline{nr_2}\bmod M$ and $r_2+\overline{n}\alpha \equiv \beta r_2\overline{nr_1}\bmod M.$
By changing variables $x\rightarrow x\overline{z}$ and $\overline{n}x\rightarrow x$,
$\overline{n}y\rightarrow y$
we obtain
\begin{equation*}
\begin{split}
\mathfrak{C}
=&g_{\chi}^{-1}g_{\bar{\chi}}^{-1}\sum_{x,y,z\in \mathbb{F}_M^{\times}}\chi(x)\bar{\chi}(y)\bar{\chi}(z)
e\left(\frac{ (r_1\overline{r_2}x+y)(\beta r_2\overline{r_1}-\alpha\overline{z})}{M}\right)
-\chi(r_2\overline{r_1}).
\end{split}
\end{equation*}
Since $\chi$ is primitive, the sums over $x$ and $y$ vanishes if
$z\equiv \alpha r_1\overline{\beta r_2}\bmod M$. Thus
we can make change of variables $(\beta r_2\overline{r_1}-\alpha\overline{z})x\rightarrow x$
and $(\beta r_2\overline{r_1}-\alpha\overline{z})y\rightarrow y$ to get
\begin{equation*}
\begin{split}
\mathfrak{C}
=&g_{\chi}^{-1}g_{\bar{\chi}}^{-1}\mathop{\sum_{z\in \mathbb{F}_M^{\times}}}_{z\not\equiv \alpha r_1\overline{\beta r_2}\bmod M} \bar{\chi}(z)
\sum_{x,y\in \mathbb{F}_M^{\times}}
\chi(x)\bar{\chi}(y)
e\left(\frac{ r_1\overline{r_2}x+y}{M}\right)
-\chi(r_2\overline{r_1})\\
=&-\chi^2(r_2\overline{r_1})\chi(\beta\overline{\alpha})-\chi(r_2\overline{r_1}).
\end{split}
\end{equation*}

Finally, if $n\neq0$ and $r_1-\overline{n}\beta=r_2+\overline{n}\alpha=0$ in $\BF_M$,
\begin{equation*}
\begin{split}
\mathfrak{C} &= \chi(n)\sum_{z\in \BF_M^\times}\bar{\chi}(r_1+z)\bar{\chi}(n+\beta\overline{z})\chi(r_2+\overline{n}\alpha+r_2\beta\overline{zn})\\
&= \chi(r_2\beta)\sum_{z\in\BF_M^\times}\bar{\chi}(r_1+z)\bar{\chi}(nz+\beta) =\chi(r_2\beta)\underset{\begin{subarray}{c}
z\in\BF_M \\ z\neq \overline{n}\beta
\end{subarray}}{\sum}\bar{\chi}(r_1-\overline{n}\beta+z)\bar{\chi}(nz)\\
&= \chi(\overline{n}r_2\beta)\underset{\begin{subarray}{c}
z\in\BF_M \\ z\neq \overline{n}\beta
\end{subarray}}{\sum}\bar{\chi}^2(z).
\end{split}
\end{equation*}
If $\chi$ is not a quadratic character, then by orthogonality of characters,  $\mathfrak{C}= -\chi(nr_2\overline{\beta})$. If $\chi$ is a quadratic character, then $\mathfrak{C}=\chi(\overline{n}r_2\beta)(M-1)$.
\end{proof}

\begin{Lemma}$\Sigma_{10}\ll (PML)^{\varepsilon}\frac{L^6M^4}{P^4N^2}$ and $\Delta_{10}\ll (PML)^{\varepsilon}\frac{L^3M^2}{P^2N}$. \label{B10}
\end{Lemma}
\begin{proof}
We recall $|r_2|<R:=N^{\varepsilon}PM/N$.
Suppose
\begin{equation}\label{restriction 3}
\begin{split}
P^2L\ll N^{1-\varepsilon}.
\end{split}
\end{equation}
Then the congruence $r_2\ell_1p_2\equiv r_1\ell_2 p_1\bmod{M}$ implies that $r_2\ell_1p_2= r_1\ell_2 p_1$. Similarly, $r_2'\ell_1p_2'= r_1'\ell_2 p_1'$. Therefore fixing $\ell_1,p_2,r_2$ fixes $\ell_2,p_1,r_1$ up to factors of $\log M$. If $\ell_1\neq\ell_2$, then the equality $r_2'\ell_1p_2'= r_1'\ell_2 p_1'$ implies $\ell_2|r_2'$. That saves a factor of $L$ in $r_2'$ sum. Further, for fixed $p_2',r_2'$, there are only $\log M$ many $p_1', r_1'$. In the case $\ell_1=\ell_2$, the previous identities become $r_2p_2=r_1p_1$. Therefore fixing $r_2, p_2$ fixes $r_1, p_1$ up to factors of $\log M$.

Finally, the congruence conditions on $r_1, r_2$ and $n$ can be combined to write
\begin{equation*}
\begin{split}
-\overline{r_1}\ell_1 p_2+\overline{r_2}\ell_2p_1+n\equiv 0\bmod p_1p_2M.
\end{split}
\end{equation*}
From \eqref{new-n-sum}, the $n$ satisfies $|n|\ll N^{1+\varepsilon}L/M$, which is smaller than the size of the modulus $p_1p_2M$. Therefore for fixed $r_i, \ell_i, p_i$, the $n$ sum is at most singleton. Similarly, $n'$ is at most a singleton. Therefore up to a factor of $(PML)^\varepsilon$,
\begin{equation*}
\begin{split}
\Sigma_{10}\ll L^5L^2\frac{1}{P^3}R\frac{1}{P^3}\frac{R}{L}M^2\ll \frac{L^6M^4}{P^4N^2}, \quad \text{ and } \quad \Delta_{10}\ll L^3\frac{1}{P^3}RM \ll \frac{L^3M^2}{P^2N}.
\end{split}
\end{equation*}

\end{proof}

\begin{Lemma}$\Sigma_{11}\ll (PML)^\varepsilon\frac{L^8M^4}{N^4P^4}\left(1+\frac{P^2}{\CN_0}\right)^2$ and $\Delta_{11}\ll (PML)^\varepsilon\frac{L^3M^2}{N^2P^2}\left(1+\frac{P^2}{\CN_0}\right)$. \label{B11}
\end{Lemma}
\begin{proof} We let the variables of summation $\ell_i, p_i, r_i, n$ to be the same as before. The expressions for $\Delta_{11}$ and $\Sigma_{11}$ are the same as $\Delta_{10}$ and $\Sigma_{10}$ with the condition $r_2\ell_1p_2\equiv r_1\ell_2p_1\bmod M$ replaced by $r_2\ell_1p_2\not\equiv r_1\ell_2p_1\bmod M$.
First let $\ell_1\neq\ell_2$. If $p_1\neq p_2$, then $(n,p_1p_2)=1$, $r_1\equiv \overline{nM}\ell_1p_2\bmod{p_1}$ and $r_2\equiv - \overline{nM} \ell_2 p_1 \bmod{p_2}$. Since $R\gg P$, these congruence conditions therefore save a factor of $O(P)$ in each $r_i$-sum. Similarly we save a factor of $O(P)$ in each $r_i'$ sum. The congruence $n\equiv0\bmod M$ (resp $n'\equiv0\bmod M$) saves a factor of at most $M$ in the $n$-sum (resp $n'$-sum). If $p_1=p_2=p$, then the congruence conditions imply $p|n$. We already have $M|n$. 
Recall from \eqref{new-n-sum} the $n$-sum satisfies $|n|\ll N^{1+\varepsilon}L/M$, which is smaller than $pM$ by our choice of $P$ and $L$. Hence we have $n=0$. The remaining congruence condition $r_1\ell_2\equiv r_2\ell_1\bmod p$ shows that fixing $r_1,\ell_2,\ell_1$ saves a factor of $O(P)$ in the $r_2$-sum. The exact same savings follow for $n'$ and $r_2'$ sums. Also, the exact same analysis as done for $p_i,r_i,n$-sums follows for the case $\ell_1=\ell_2$.  Therefore up to a factor of $(PML)^\varepsilon$,
\begin{equation*}
\begin{split}
\Sigma_{11}&\ll L^8\bigg[\frac{1}{P^2}\frac{R^2}{P^2}\left(1+\frac{P^2}{\CN_0}\right)\bigg]^2 + L^8\bigg[\frac{1}{P^3}\frac{R^2}{P}\bigg]^2 \ll \frac{L^8M^4}{N^4P^4}\left(1+\frac{P^2}{\CN_0}\right)^2,\\
\Delta_{11} &\ll L^3\bigg[\frac{1}{P^2}\frac{R^2}{P^2}\left(1+\frac{P^2}{\CN_0}\right)\bigg] + L^3\bigg[\frac{1}{P^3}\frac{R^2}{P}\bigg] \ll \frac{L^3M^2}{N^2P^2}\left(1+\frac{P^2}{\CN_0}\right).
\end{split}
\end{equation*}
\end{proof}

\begin{Lemma}$\Sigma_{20}\ll (PML)^\varepsilon\frac{L^6M^4}{P^4N^2}$ and $\Delta_{20}\ll (PML)^\varepsilon\frac{L^3M^2}{P^2N}$. \label{B20}
\end{Lemma}
\begin{proof}
The congruence conditions on $r_1, r_2$ and $n$ can be combined to write
\begin{equation*}
-\ell_1p_2+nr_1\equiv0\bmod p_1M \quad \text{ and } \quad \ell_2p_1+nr_2\equiv0\bmod p_2M.
\end{equation*}
By \eqref{new-n-sum}, we have $|nR|\ll N^{\varepsilon}PL<PM^{1-\varepsilon}$. The congruence conditions therefore give equalities 
\begin{equation*}
\begin{split}
&n r_1 = \ell_1 p_2 \quad \text{ and } \quad n r_2 = -\ell_2 p_1, \\
&n' r_1' = \ell_1 p_2' \quad \text{ and } \quad n' r_2' = -\ell_2 p_1'.
\end{split}
\end{equation*}
Note that $n=\ell_1p_2/r_1 = -\ell_2 p_1/r_2$ implies $\ell_1p_2r_2=-\ell_2p_1r_1$. Therefore fixing $\ell_1,p_2,r_2$ fixes $\ell_2,p_1,r_1$ up to factors of $\log M$. Similarly, $n'=\ell_1p_2'/r_1' = -\ell_2 p_1'/r_2'$, so that $\ell_1p_2'r_2'=-\ell_2p_1'r_1'$. If $\ell_1\neq\ell_2$, then $\ell_2|r_2'$. That saves a factor of $L$ in $r_2'$ sum. Moreover for fixed $\ell_1,p_2',r_2'$, there are only $\log M$ many $p_1', r_1'$. Finally the identities $n r_1 = \ell_1 p_2$ and $n' r_1' = \ell_1 p_2'$ fix $n$ and $n'$. 

In the case $\ell_1=\ell_2$, the previous identities become $r_2p_2=-r_1p_1$. Therefore fixing $r_2, p_2$ fixes $r_1, p_1$ up to factors of $\log M$. Finally the identity $n r_1 = \ell p_2$ fixes $n$. Therefore up to a factor of $(PML)^\varepsilon$,
\begin{equation*}
\begin{split}
\Sigma_{20}\ll L^5L^2\frac{1}{P^3}R\frac{1}{P^3}\frac{R}{L}M^2\ll \frac{L^6M^4}{P^4N^2}, \quad \text{ and } \quad \Delta_{20} \ll L^3\frac{1}{P^3}RM \ll \frac{L^3M^2}{P^2N}.
\end{split}
\end{equation*}

\end{proof}

\begin{Lemma}$\Sigma_{21}\ll (PML)^\varepsilon\frac{L^8M^5}{N^4P^4}\left(1+\frac{P^2M}{\CN_0}\right)^2$ and $\Delta_{21}\ll (PML)^\varepsilon\frac{L^3M^{5/2}}{N^2P^2}\left(1+\frac{P^2M}{\CN_0}\right)$. \label{B21}
\end{Lemma}
\begin{proof}
When $p_1\neq p_2$, the congruence $-\overline{r_1}\ell_1 p_2+\overline{r_2}\ell_2p_1+n\equiv 0(p_1p_2)$ implies that $(n,p_1p_2)=1$. Moreover, for fixed $n$, $p_i$ and $\ell_i$, $i=1,2$,
\begin{equation*}
\begin{split}
r_1\equiv \overline{n}\ell_1p_2\bmod{p_1} \quad \text{ and } \quad r_2\equiv -\overline{n}\ell_2p_1\bmod{p_2}.
\end{split}
\end{equation*}
These congruence conditions save a factor of $P$ in each $r_i$-sum. In case $p_1=p_2=p$, the congruence condition shows $p|n$. Moreover, $-\overline{r_1}\ell_1+\overline{r_2}\ell_2+n/p \equiv 0 \bmod p$. Therefore fixing $r_1, \ell_1, \ell_2, n$ saves $P$ in $r_2$-sum. Similarly we get saving of $P$ for each of the $n'$ and $r_2'$ sums. Also, the exact same analysis as done for $p_i,r_i,n$-sums follows for the case $\ell_1=\ell_2$. Therefore up to a factor of $(PML)^\varepsilon$,
\begin{equation*}
\begin{split}
\Sigma_{21}&\ll  L^8\bigg[\frac{1}{P^2}\frac{R^2}{P^2}\left(1+\frac{P^2M}{\CN_0}\right)M^{1/2}\bigg]^2 + L^8\bigg[\frac{1}{P^2}\frac{R^2}{P^2}\left(1+\frac{PM}{\CN_0}\right)M^{1/2}\bigg]^2 \ll \frac{L^8M^5}{N^4P^4}\left(1+\frac{P^2M}{\CN_0}\right)^2, \\
\Delta_{21} &\ll  L^3\bigg[\frac{1}{P^2}\frac{R^2}{P^2}\left(1+\frac{P^2M}{\CN_0}\right)M^{1/2}\bigg] + L^3\bigg[\frac{1}{P^2}\frac{R^2}{P^2}\left(1+\frac{PM}{\CN_0}\right)M^{1/2}\bigg] \ll \frac{L^3M^{5/2}}{N^2P^2}\left(1+\frac{P^2M}{\CN_0}\right).
\end{split}
\end{equation*}
\end{proof}

\subsection*{Acknowledgments.}
We are very grateful to Paul D. Nelson for many valuable comments. We are also very thankful to the referee for his or her very careful reading and detailed comments of the manuscript. K. A. and Y. L. thank their advisor R. H. for introducing them into the project and explaining the idea.  Q. S. was partially supported by IRT16R43 and CSC.

\bibliographystyle{amsplain}
\bibliography{ref}

\providecommand{\bysame}{\leavevmode\hbox to3em{\hrulefill}\thinspace}
\providecommand{\MR}{\relax\ifhmode\unskip\space\fi MR }
\providecommand{\MRhref}[2]{%
  \href{http://www.ams.org/mathscinet-getitem?mr=#1}{#2}
}
\providecommand{\href}[2]{#2}
\begin{thebibliography}{10}

\bibitem{Adolphson-Sperber}
A.~Adolphson and S.~Sperber, \emph{Twisted exponential sums and {N}ewton
  polyhedra}, J. Reine Angew. Math. \textbf{443} (1993), 151--177. \MR{1241131}

\bibitem{Blomer-Harcos}
V.~Blomer and G.~Harcos, \emph{Hybrid bounds for twisted {$L$}-functions}, J.
  Reine Angew. Math. \textbf{621} (2008), 53--79. \MR{2431250}

\bibitem{Blo-Har-Mic}
V.~Blomer, G.~Harcos, and Ph. Michel, \emph{A {B}urgess-like subconvex bound
  for twisted {$L$}-functions}, Forum Math. \textbf{19} (2007), no.~1, 61--105,
  Appendix 2 by Z. Mao. \MR{2296066}

\bibitem{Burgess}
D.~A. Burgess, \emph{On character sums and {$L$}-series. {II}}, Proc. London
  Math. Soc. (3) \textbf{13} (1963), 524--536. \MR{0148626}

\bibitem{Bykovskii}
V.~A. Bykovski\u\i, \emph{A trace formula for the scalar product of {H}ecke
  series and its applications}, Zap. Nauchn. Sem. S.-Peterburg. Otdel. Mat.
  Inst. Steklov. (POMI) \textbf{226} (1996), no.~Anal. Teor. Chisel i Teor.
  Funktsi\u\i . 13, 14--36, 235--236. \MR{1433344}

\bibitem{Conrey-Iwaniec}
J.~B. Conrey and H.~Iwaniec, \emph{The cubic moment of central values of
  automorphic {$L$}-functions}, Ann. of Math. (2) \textbf{151} (2000), no.~3,
  1175--1216. \MR{1779567}

\bibitem{DFI}
W.~Duke, J.~Friedlander, and H.~Iwaniec, \emph{Bounds for automorphic
  {$L$}-functions}, Invent. Math. \textbf{112} (1993), no.~1, 1--8.
  \MR{1207474}

\bibitem{FKM}
\'E. Fouvry, E.~Kowalski, and Ph. Michel, \emph{Algebraic twists of modular
  forms and {H}ecke orbits}, Geom. Funct. Anal. \textbf{25} (2015), no.~2,
  580--657. \MR{3334236}

\bibitem{Fu1}
L.~Fu, \emph{Weights of twisted exponential sums}, Math. Z. \textbf{262}
  (2009), no.~2, 449--472. \MR{2504886}

\bibitem{HN17}
R.~Holowinsky and P.~D. Nelson, \emph{Subconvex bounds on {$\rm GL_3$} via
  degeneration to frequency zero}, Math. Ann. \textbf{372} (2018), no.~1-2,
  299--319. \MR{3856814}

\bibitem{KMV2002}
E.~Kowalski, Ph. Michel, and J.~VanderKam, \emph{Rankin-{S}elberg
  {$L$}-functions in the level aspect}, Duke Math. J. \textbf{114} (2002),
  no.~1, 123--191. \MR{1915038}

\bibitem{Lin}
Y.~Lin, \emph{{Bounds for twists of $\rm GL(3)$ {$L$}-functions}},
  arXiv:1802.05111 (2018).

\bibitem{LMS}
Y.~Lin, Ph. Michel, and W.~Sawin, \emph{Algebraic twists of {$\rm GL_3\times
  \rm GL_2$} {$L$}-functions}, arXiv:1912.09473 (2019).

\bibitem{LiuWangYe}
J.~Liu, Y.~Wang, and Y.~Ye, \emph{A proof of {S}elberg's orthogonality for
  automorphic {$L$}-functions}, Manuscripta Math. \textbf{118} (2005), no.~2,
  135--149. \MR{2177681}

\bibitem{Molteni}
G.~Molteni, \emph{Upper and lower bounds at {$s=1$} for certain {D}irichlet
  series with {E}uler product}, Duke Math. J. \textbf{111} (2002), no.~1,
  133--158. \MR{1876443}

\bibitem{Mun1}
R.~Munshi, \emph{The circle method and bounds for {$L$}-functions---{IV}:
  {S}ubconvexity for twists of {$\rm GL(3)$} {$L$}-functions}, Ann. of Math.
  (2) \textbf{182} (2015), no.~2, 617--672. \MR{3418527}

\bibitem{Mun2}
\bysame, \emph{Twists of {$\rm GL(3)$} {$L$}-functions}, arXiv:1604.08000
  (2016).

\bibitem{Munshi17}
\bysame, \emph{A note on {B}urgess bound}, Geometry, algebra, number theory,
  and their information technology applications, Springer Proc. Math. Stat.,
  vol. 251, Springer, Cham, 2018, pp.~273--289. \MR{3880392}

\bibitem{Petrow-Young1}
I.~Petrow and M.~P. Young, \emph{The {W}eyl bound for {D}irichlet
  {$L$}-functions of cube-free conductor}, arXiv:1811.02452 (2018).

\bibitem{Petrow-Young2}
\bysame, \emph{The fourth moment of {D}irichlet {$L$}-functions along a coset
  and the {W}eyl bound}, arXiv:1908.10346 (2019).

\end{thebibliography}
\end{document}